\documentclass[12pt]{amsart}
\usepackage{amscd} 
\usepackage{amsfonts} 
\usepackage{amssymb} 
\usepackage{latexsym} 
\usepackage[arrow, matrix, curve]{xy}

\newcommand{\ncm}{\newcommand}

\newtheorem{theorem}{Theorem}[section]
\newtheorem{prop}[theorem]{Proposition}
\newtheorem{lemma}[theorem]{Lemma}
\newtheorem{cor}[theorem]{Corollary}
\newtheorem{lem&def}[theorem]{Lemma \& Definition}
\newtheorem{definition}[theorem]{Definition}
\newtheorem{example}[theorem]{Example}

\def\M{\mathcal{M}}

\def\C{\mathbb{C}\,} 
\def\Z{\mathbb{Z}\,} 
 
\def\N{\mathbb{N}\,}

\ncm{\End}{\mbox{\rm End}\,}
\def\Ann{\mbox{\rm Ann}}
\def\Hom{\mbox{\rm Hom}}

\def\|{\, | \,}

\def\id{\mbox{\rm id}}

\def\into{\hookrightarrow}

\def\bra{\langle}
\def\ket{\rangle}

\ncm{\rarr}[1]{\stackrel{#1}{\longrightarrow}}
\ncm{\larr}[1]{\stackrel{#1}{\longleftarrow}}
\def\cop{\Delta}

\def\eps{\varepsilon}

\def\-2{_{(-2)}}
\def\-1{_{(-1)}}
\def\0{_{(0)}}
\def\1{_{(1)}}
\def\2{_{(2)}}
\def\3{_{(3)}}
\def\n{_{(n)}}

\def\du1{\hat 1}

\begin{document}
\title[Subgroup depth via  algebraic quotient modules]{Algebraic quotient modules \\ and Subgroup depth}
\author[A.~Hernandez, L.~Kadison and C.J.~Young]{Alberto Hernandez, Lars Kadison and Christopher  Young} 
\address{Departamento de Matematica \\ Faculdade de Ci\^encias da Universidade do Porto \\ 
Rua Campo Alegre 687 \\ 4169-007 Porto} 
\email{ahernandeza079@gmail.com, lkadison@fc.up.pt} 
\thanks{}
\date{} 

\begin{abstract}
In \cite{K2013} it was shown  that subgroup depth may be computed from the permutation module of the left or right cosets: this holds more generally for a Hopf subalgebra, from which we note in this paper that finite depth of a Hopf subalgebra $R \subseteq H$ is equivalent to the $H$-module coalgebra $Q = H/R^+H$ representing an algebraic element in the Green ring of $H$ or $R$.  This approach shows that 
subgroup depth and the subgroup depth of the corefree quotient lie in the same closed interval of length one.  
We also establish a previous claim that the problem of determining if $R$ has finite depth in $H$ is equivalent to determining if $H$ has
finite depth in its smash product $Q^* \# H$.  A necessary condition is obtained for finite depth from stabilization of a descending chain of annihilator ideals of tensor powers of $Q$.   
As an application of  these topics to a centerless finite group $G$, we prove  that the minimum depth of its group $\C$-algebra in the Drinfeld double $D(G)$ is an odd integer, which determines the least tensor power of  the adjoint representation $Q$ that is a faithful  $\C G$-module.  
\end{abstract} 
\maketitle

\section{Introduction and Preliminaries}
\label{one}
One of the themes in Feit's book on group representations is that a permutation module is an algebraic module
(based on Mackey's theorem) \cite[IX.3.2]{Fe}.  Thus given a subgroup $H \leq G$ of a finite group and $k$ an arbitrary field,
the permutation module of cosets satisfies a polynomial equation in either Green ring $A(G)$ or
$A(H)$.  
It is then interesting to ask when the generalized quotient $Q = H/R^+H$ of a finite Hopf subalgebra pair $R \subseteq H$ is an algebraic
module (equivalently, \cite[p.\  250]{BDK}).  In case the isoclass of $Q$ is an algebraic element in the representation ring of $R$, the degree of a minimum polynomial  of $Q$ in $A(R)$ is related to the subalgebra depth of $R$ in $H$ (the minimum even depth of $R \subseteq H$ is  twice this degree).  By depth we mean the (induction-restriction) depth defined  and studied in the papers \cite{BDK, BKK}, and further studied in light of  modular and ordinary representation theory of groups and semisimple algebras in \cite{BK, BK2, BuK, D, F, FKR, HHP, LK2012}. In this paper we carry the connections
established in \cite{K2013} between subalgebra depth of $R \subseteq H$
and module depth of $Q$ further in a few directions; we show that the techniques previously established for computing depth of $R$ in $H$ compute the number $\ell_Q$, equal to the least tensor power that is faithful, of a conditionally faithful projective quotient module $Q$  (see Theorem~\ref{theorem-Burn},  Corollary~\ref{cor-precisedetermination} and Example~\ref{ex-Pass}).
 We also show that subgroup depth satisfies $|d(H,G) - d(H/N,G/N)| \leq 1$ where $N$ is a normal
subgroup of a finite group $G$ contained in a subgroup $H$ (see Section 2.1).  

The topics and layout of the paper are as follows. After an introduction of terminology and previously established facts for subalgebra and subgroup depth in Section 1,  a necessary condition for finite depth for a subalgebra pair $B \subseteq A$ of finite-dimensional $k$-algebras  is given in Section~\ref{four} in the form of 
a matrix inequality between products of matrices of induction and of restriction,
which are related by the Cartan matrices of $A$ and $B$ if $k$ is algebraically closed. 
We review in Section~3 the simplifications for depth of Hopf subalgebra $R  \subseteq H$ in terms of a module depth of
a quotient module $Q := H/R^+H$, based on \cite{K2013}. Noting that a module $W$ over a finite-dimensional Hopf algebra with Hopf ideal $I$ in its annihilator ideal,
has module depth $d(W)$  the same over $H$ or the quotient Hopf algebra $H/I$, we conclude from the close relationship of $d(W)$ and $d(R,H)$
that subgroup depth differs by at most one with the subgroup depth of its corefree quotient (Proposition~\ref{prop-core}).  In Section~\ref{two} we obtain a necessary condition for finite depth involving stabilization of a descending chain of annihilator ideals of tensor powers of $Q$ (Proposition~\ref{prop-necessary}). 
If $R$ is a Hopf subalgebra, we establish a previous claim \cite[4.14, 5.5]{K2013} that the problem of determining if $R$ has finite depth in $H$ is equivalent to determining if $H$ has
finite depth in its smash product $Q^* \# H$ (Theorem~\ref{th-young} and
 Corollary~\ref{cor-equivalence}).  We note that
the minimum depth of a finite group $\C$-algebra in its Drinfeld double is an odd integer determined by the least tensor power of $Q$ that is faithful (Section~\ref{three} and Corollary~\ref{cor-precisedetermination}).

\subsection{Preliminaries on subalgebra depth} Let $A$ be a unital associative algebra over a field $k$. In this paper we assume all algebras and modules to be finite-dimensional vector spaces (although most of the definitions and facts in this subsection do fine without this assumption \cite{LK2012}). The category of finite-dimensional modules over $A$ will be denoted
by $\M_A$.  
Two modules $M_A$ and $N_A$ are \textit{similar} (or H-equivalent) if $M \oplus * \cong q \cdot N := N \oplus \cdots \oplus N$ ($q$ times) and $N \oplus * \cong r \cdot M$
for some $r,q \in \N$.  This is briefly denoted by $M \| q \cdot N$ and $N \| r \cdot M$ for some $q,r \in \N$  $\Leftrightarrow$ $M \sim N$. It is well-known that similar modules have Morita equivalent endomorphism rings.  

 Let $B$ be a subalgebra of $A$ (always supposing $1_B = 1_A$).  Consider the natural bimodules ${}_AA_A$, ${}_BA_A$, ${}_AA_B$ and
${}_BA_B$ where the last is a restriction of the preceding, and so forth.  Denote the tensor powers
of ${}_BA_B$ by $A^{\otimes_B n} = A \otimes_B \cdots \otimes_B A$ for $n = 1,2,\ldots$, which is also a natural bimodule over  $B$ and $A$ in any one of four ways;     set $A^{\otimes_B 0} = B$ which is only a natural $B$-$B$-bimodule.  

\begin{definition}
\label{def-depth}
If $A^{\otimes_B (n+1)}$ is similar to $A^{\otimes_B n}$ as $X$-$Y$-bimodules,  one states that $B \subseteq A$ has  
\begin{itemize}
\item \textbf{depth} $2n+1$ if $X = B = Y$;
\item \textbf{left depth} $2n$ if $X = B$ and $Y = A$;
\item \textbf{right depth} $2n$ if $X = A$ and $Y = B$;
\item \textbf{h-depth} $2n-1$ if $X = A = Y$,
\end{itemize}
valid for even depth and h-depth if $n \geq 1$ and for odd depth if $n \geq 0$. 

Note that if $B \subseteq A$ has h-depth $2n-1$, the subalgebra has (left or right) depth $2n$ by restriction of modules.  Similarly, if $B \subseteq A$ has depth $2n$, it has depth $2n+1$.  If $B \subseteq A$ has depth $2n+1$, it has depth $2n+2$ by tensoring either $-\otimes_B A$ or $A \otimes_B -$ to $A^{\otimes_B (n+1)} \sim A^{\otimes_B n}$.     Similarly, if $B \subseteq A$ has left or right depth $2n$, it has h-depth $2n+1$.  Denote the \textbf{minimum depth}
of $B \subseteq A$ (if it exists) by $d(B,A)$ \cite{BDK}.  Denote the \textbf{minimum h-depth} of $B \subseteq A$ by $d_h(B,A)$ \cite{LK2011}.  Note that
$d(B,A) < \infty$ if and only if 
$d_h(B,A) < \infty$; more precisely, 
\begin{equation}
\label{eq: in}
d_h(B,A) -2 \leq d(B,A) \leq d_h(B,A) + 1
\end{equation}
 if either is finite.  
\end{definition}

 For example, $B \subseteq A$ has depth $1$ iff ${}_BA_B$ and ${}_BB_B$ are similar \cite{BK2, LK2012}.  In this case, it is easy to show that $A$ is algebra isomorphic to $B \otimes_{Z(B)} A^B$ where
$Z(B), A^B$ denote the center of $B$ and centralizer of $B$ in $A$.

Another example, $B \subset A$ has right depth $2$ iff ${}_AA_B$
and ${}_A A \otimes_B A_B$ are similar.  If $A = \C G$ is a group algebra of a  finite group $G$ and $B = \C H$ is a group algebra of a subgroup $H$ of $G$, then $B \subseteq A$ has right depth $2$ iff $H$ is a normal subgroup of $G$ iff $B \subseteq A$ has left depth $2$ \cite{KK}; a similar statement is true for a Hopf subalgebra $R \subseteq H$ of finite index and over any field \cite{BK}.  Ring extensions having depth $2$ are bialgebroid-Galois extensions (if a generator condition on $A_B$ is met); conversely, most Galois extensions have depth $2$ (except the coalgebra-Galois extensions, which include all examples of Hopf subalgebras mentioned below) \cite{BW}.   

Note that $A^{\otimes_B n} \| A^{\otimes_B (n+1)}$ for all $n \geq 2$
and in any of the four natural bimodule structures: one applies $1$ and multiplication to obtain a split monic, or split epi oppositely. For three of the bimodule structures, it is true for $n =1$;  as $A$-$A$-bimodules, equivalently $A \| A \otimes_B A$ as $A^e$-modules, this is the separable extension condition on $B \subseteq A$.  
But $A \otimes_B A \| q \cdot A$ as $A$-$A$-bimodules for some $q \in \N$
is the H-separability condition and implies $A$ is a separable extension of $B$ \cite{K}.  Somewhat similarly, ${}_BA_B \| q  \cdot {}_BB_B$ implies
${}_BB_B \| {}_BA_B$ \cite{LK2012}. It follows that subalgebra depth and h-depth may be equivalently defined by replacing the similarity bimodule conditions for depth and h-depth in Definition~\ref{def-depth} with the corresponding bimodules on 
\begin{equation}
\label{eq: def}
A^{\otimes_B (n+1)} \, \| \, q \cdot A^{\otimes_B n}
\end{equation}
 for some positive integer $q$ \cite{BDK, LK2011, LK2012}.  

Subgroup depth \cite{BKK, BDK} of a finite group-subgroup pair $G \geq H$ over a field
is defined to be the minimum depth $d_k(H,G) := d(kH, kG)$ of the associated finite
group algebra extension.  In \cite{BDK} it was shown that $d_k(H,G)$ is the same for
any field $k$ of the same characteristic; therefore we write $d_0(H,G)$ if $k$ has characteristic zero,
and $d_p(H,G)$ if it has prime $p$ characteristic.  The inequality  $d_0(H,G) \leq d_p(H,G)\leq 2|G: N_G(H) |$ is noted in \cite[4.5]{BDK}, which also extends this inequality to include group algebras over commutative rings, and defines a \textit{combinatorial depth} $d_c(H,G)$ using only the group-theoretic notions
of bisets and their monomorphisms. 
For example, for the permutation groups $\Sigma_n < \Sigma_{n+1}$
and their corresponding group algebras $B \subseteq A$ over any
commutative ring $K$, one has depth $d(B,A) = 2n-1 = d_c(\Sigma_n, \Sigma_{n+1})$  \cite{BDK}, which also notes  from other examples strict inequalities between the various depths. Depths of subgroups in $PGL(2,q)$, twisted group algebra extensions and Young subgroups of $\Sigma_n$ are computed in \cite{F, D, FKR}. 

   If $B$ and $A$ are semisimple complex algebras, the minimum odd depth is computed from powers of an order $r$ symmetric matrix with nonnegative entries $\mathcal{S} := MM^T$ where $M$ is the inclusion matrix (or induction-restriction matrix) 
$K_0(B) \rightarrow K_0(A)$ and $r$ is the number of irreducible representations of $B$ in a basic set of $K_0(B)$; the depth is $2n+1$ if $\mathcal{S}^n$ and $\mathcal{S}^{n+1}$ have an equal number of zero entries \cite{BKK}. 
The minimum even depth is similarly computed considering the zero entries of the  powers $\mathcal{S}^mM$.  The (overall) minimum depth of $B \subseteq A$
is equivalently computed as the least $n \in \N$ for which the induction-restriction matrix $M$ satisfies the   matrix inequality 
\begin{equation}
\label{eq: matrixineqdepth}
M^{n+1} \leq q M^{n-1}
\end{equation}
(for some $q \in \N$, each $(i,j)$-entry, and denoting $M^0 = I_r$, $M^{2m} = \mathcal{S}^m$, $M^{2m+1} =\mathcal{S}^mM$, each $m \in \N$) \cite{BKK}. 

Similarly, the minimum h-depth of $B\subseteq A$ is computed from
powers of an order $s$ symmetric matrix $\mathcal{T} = M^TM$, where $s$ is the rank of $K_0(A)$, and the power $n$ at
which the number of zero entries of $\mathcal{T}^n$ stabilizes \cite{LK2012}. 
From the matrix power definitions,  it follows that a subalgebra pair of semisimple algebras $B \subseteq A$ over a field of characteristic zero  always has finite depth. (In characteristic $p$ finite depth holds if one of $B,A$ is
a separable algebra \cite[Corollary 2.2]{KY}.)  

%%%%%%%%%%%%%%%%%%%%%%%%%%%%%%%%%%%%%%%%%%%%%%%

\section{Depth of subalgebras  projective in a finite-dimensional algebra} 
\label{four}
In this section we look more closely at the matrix inequality condition~(\ref{eq: matrixineqdepth}) for depth explained in Section~1, but for a subalgebra pair of finite-dimensional algebras.  In the presence of radical ideals, the induction-restriction matrix $M$ splits into two matrices related by a pair of Cartan matrices.  

Let $A$ be a finite-dimensional algebra over a field $k$.  Denote the principal right $A$-modules, or projective indecomposables of $A$, by $P_1, \ldots, P_s$. (We sometimes confuse objects and their isoclasses
for the sake of brevity.)  Let $J$ denote the radical ideal of $A$. Then
each $P_i$ is the projective cover of 
$P_i / P_i J := S_i$, the simple $A$-modules where $i = 1,\ldots,s$.  Recall that the Cartan matrix $C$ 
of $A$ is an $s \times s$-matrix of nonnegative entries whose rows give the multiplicity of each simple $S_j$ in the composition factors of $P_i$; one may view $C$ as the matrix of  a linear mapping $K_0(A) \rightarrow G_0(A)$ corresponding to sending a projective into a sum of its simple composition factors
with multiplicity. Recall that $K_0(A) \cong Z^s$ is a free abelian group on the basis $P_1, \ldots, P_s$,
such that a projective $X$ in $K_0(A)$ is a nonnegative sum of the $P_i$ corresponding to its Krull-Schmidt decomposition; also recall that $G_0(A) \cong Z^s$ is the free abelian group on the basis $S_1,\ldots,S_s$ (the Grothendieck group of $A$) such that a module $Y$ in $G_0(A)$ is a nonnegative sum of the $S_i$ corresponding to the multiplicity of its composition factors.  If $k$ is an algebraically closed field, $\dim_k \Hom_A (P_i, X)$ equals the multiplicity of (the isomorphism class of) $S_i$
as a composition factor in a finite-dimensional module $X$ \cite[p.\  45]{ARS}:  in this case,
the Cartan matrix entry $c_{ij} = \dim \Hom_A (P_i, P_j)$ for each $i,j = 1,\ldots,s$. 

Suppose $B \subseteq A$ is a subalgebra of $A$ such that the natural module $A_B$ is projective.
Denote the projective indecomposables of $B$ by $Q_1, \ldots Q_r$, the Cartan matrix of $B$
by $D$, which has entries $d_{ij} = \dim \Hom_B(Q_i,Q_j)$ in case $k$ is algebraically closed.  

Of interest to us are two $r \times s$-matrices with nonnegative entries. (For both matrices, we use the Krull-Schmidt Theorem for finite length modules of Artin algebras.)  First define 
the \textit{matrix of restriction} $M$ with entries given by $m_{ij}$ defined by
\begin{equation}
\label{eq: restrictionmatrix}
P_j \downarrow_B \cong \oplus_{i=1}^r m_{ij} \cdot Q_i 
\end{equation}
since each projective $A$-module restricts to a projective $B$-module by the hypothesis that $A_B$ is projective. Secondly, define the \textit{matrix of induction} for the subalgebra $B \subseteq A$ as the $r \times s$-matrix $N$ with row entries  $n_{ij} \in \N$
given by inducing each of the projective indecomposable $B$-modules,
\begin{equation}
\label{eq: inductionmatrix}
Q_i \otimes_B A \cong \oplus_{j=1}^s n_{ij} \cdot P_j
\end{equation} 

\begin{lemma}
\label{lemma-cartaninductionrestriction}
Suppose $k$ is algebraically closed.  Then the matrices of restriction $M$ and induction $N$
are related by 
\begin{equation}
\label{eq: cartan}
D M = N C
\end{equation}
where $C$ and $D$ denote the Cartan matrices of $A$ and $B$, respectively.
\end{lemma}
\begin{proof}
From the Hom-Tensor adjoint relation it follows that $$\Hom_A(Q_i \otimes_B A, P_j) \cong
\Hom_B(Q_i, P_j\downarrow_B)$$ \cite{K}. Substitution of Eqs.~(\ref{eq: inductionmatrix}) and~(\ref{eq: restrictionmatrix}) reduces to 
$$ \oplus_{k=1}^s n_{ik} \cdot \Hom_A(P_k,P_j) \cong \oplus_{q=1}^r m_{qj} \cdot \Hom_B(Q_i,Q_q). $$
Taking the dimension of both sides yields $\sum_{k=1}^s n_{ik}c_{kj} = \sum_{q=1}^r m_{qj} d_{iq}$. for each $i = 1,\ldots, r$, $j = 1,\ldots,s$, from which the lemma follows.
\end{proof}

\begin{example}
\label{example-semisimplepair}
\begin{rm}
Suppose $A$ and $B$ are semisimple algebras with $B$ a subalgebra of $A$. Then $P_i = S_i$
so that the Cartan matrix  of $A$ is the identity matrix, $C = I_s$; similarly, the Cartan matrix of $B$
satisfies $D = I_r$.   It follows from the lemma that if the ground field $k$ is algebraically closed, 
$M = N$, which is then the induction-restriction matrix studied in \cite{BKK} for $k$ additionally of characteristic zero, or the induction-restriction table studied in \cite{AB} for subgroup pairs of finite complex group algebras.  That $M = N$ also follows from the proof of Lemma~\ref{lemma-cartaninductionrestriction}
by applying Schur's Lemma for algebraically closed fields to $\Hom_A(S_i,S_j) \cong k \delta_{ij}$ and similarly
$\dim \Hom_B(Q_i,Q_j) = \delta_{ij}$. 
\end{rm}
\end{example}

\begin{example}
\label{example-triangular}
\begin{rm}
Let $A = T_n(k)$ be the upper triangular $n \times n$-matrices over an algebraically closed field $k$.
Let $B = \mbox{Diag}_n(k)$ the diagonal matrices of order $n$, a semisimple subalgebra of $A$.
The Cartan matrix $D = I_n$ is immediate.  
Let $J $ denote the radical ideal of $A$, so that the obvious algebra epimorphism $A \rightarrow B$
is equal to the canonical epi $A \rightarrow A/J$.   Denote the simples of $A$ by $S_1,\ldots,S_n$
which are then also the simples of $B$ by restriction. Thus $Q_i = S_i\downarrow_B$ for each $i = 1,\ldots,n$.  The projective indecomposable right $A$-modules
are given in terms of matrix units by $P_1 = e_{11} A, \ldots, P_n = e_{nn}A$, which are the projective covers of $S_1,\ldots,S_n$, respectively.  Then the matrix of induction from $B$ to $A$ is $N = I_n$, since $S_i \otimes_B A \cong P_i$ is immediate from writing $S_i = B e_{ii}$. 

 The composition series of $P_i$ is given by
$P_i \supset P_i J \supset P_i J^2 \supset \cdots \supset P_i J^{n-i+1} = \{0 \}$ with simple factors
$ P_i / P_i J \cong S_i, P_i J / P_i J^2 \cong S_{i+1}, $ and so forth, obtaining the Cartan matrix
$C = \sum_{i \leq j} e_{ij}$ for $A$.  Restriction of the principal modules, $P_1 \downarrow_B \cong
Q_1 \oplus \cdots \oplus Q_n$ is clear from writing $P_1 = \sum_{j=1}^n e_{1j}k$ and the matrix unit equations $e_{ij} e_{qk} = \delta_{jq} e_{ik}$.  Similarly, $P_i\downarrow_B \cong Q_i \oplus \cdots \oplus Q_n$, whence the restriction matrix of $B \subset A$ is $M = \sum_{i \leq j} e_{ij}$.  
Indeed $M = C$ as implied by Lemma~\ref{lemma-cartaninductionrestriction}. 
\end{rm}
\end{example}

The theorem below does not require that $k$ is algebraically closed.  Set the zeroeth power of a square matrix equal to the identity matrix. Denote the transpose of a matrix $X$ by $X^T$. 
\begin{theorem}
\label{theorem-matrixinequality}
Suppose $B \subseteq A$ is a subalgebra pair of finite-dimensional $k$-algebra with
$A_B$ assumed projective.  If the subalgebra $B \subseteq A$ has left depth $2n$ (respectively, depth $2n+1$), then
\begin{equation}
\label{eq: even}
(MN^T)^n M \leq t(MN^T)^{n-1}M \ \  (\mbox{resp.}\ (MN^T)^{n+1} \leq t(MN^T)^n )
\end{equation}
for some $t \in \N$.  
\end{theorem}
\begin{proof}
Suppose $B \subseteq A$ has depth $1$.  Then for some $B$-$B$-bimodule $W$,
we have 
\begin{equation}
\label{eq: start}
{}_BA_B \oplus {}_BW_B \cong t \cdot {}_BB_B
\end{equation}
 for some positive $t \in \N$.    Tensoring Eq.~(\ref{eq: start}) to
the right $B$-projective indecomposable $Q_i$, one obtains after a standard cancellation, 
\begin{equation}
\label{eq: intermediate}
Q_i \otimes_B A \downarrow_B \oplus \ Q_i \otimes_B W_B \cong t \cdot Q_i.
\end{equation}
By the Krull-Schmidt Theorem, there is $w_i \in \N$ such that $Q_i \otimes_B W_B \cong w_i \cdot Q_i$  for each $i = 1,\ldots,r$; and using Eqs.~(\ref{eq: inductionmatrix}) and~(\ref{eq: restrictionmatrix}),
$Q_i \otimes_B A_B \cong (\sum_{j=1}^s n_{ij} m_{ij}) \cdot Q_i$.  It follows from $w_i \geq 0$ 
and Eq.~(\ref{eq: intermediate})
 that $MN^T \leq tI_r$.  The rest of the proof is a similar application of the matrices of restriction and induction to the characterization of depth $2n, 2n+1$ subalgebra in Eq.~(\ref{eq: def}).
\end{proof}
In \cite[2.1, 3.5]{BKK} or Eq.~(\ref{eq: matrixineqdepth}) the matrix inequality~(\ref{eq: even}) with $M = N$ \textit{characterizes} a depth $n$ 
semisimple complex algebra-subalgebra pair $B \subseteq A$.  
\begin{example}
\label{example-counter}
\begin{rm}
Example~\ref{example-triangular}  provides a counterexample to the converse for Theorem~\ref{theorem-matrixinequality}.  Recall that $A$ is the upper triangular matrix algebra and $B$ is the subalgebra of diagonal matrices. Then the minimum depth $d(B,A)$ is computed in \cite{KY} as the semisimple subalgebra of quiver vertices within the path algebra for the quiver $$1 \rightarrow 2 \rightarrow \cdots \rightarrow n-1 \rightarrow n.$$
The depth satisfies $d(B,A) = 3$ as a corollary of \cite[Section 6, first paragraph]{KY}.  However,
we computed the $n \times n$ restriction matrix $M = \sum_{i \leq j} e_{ij}$ in terms of matrix units, and
the induction matrix $N = I_n$.  It follows that $MN^T = M$, all of whose powers satisfy
$M^s \leq tM^{s-1}$ for integers $s \geq 2$ and some positive $t \in \N$ (depending on $s$), since  the set of upper triangular matrices with only positive entries is closed under matrix multiplication.
In particular, the subalgebra $B$ does not have depth two in $A$, although it satisfies the depth two
matrix inequality $M^2 \leq nM$ (taking $t = n$) in Theorem~\ref{theorem-matrixinequality}.
\end{rm}
\end{example}

%%%%%%%%%%%%%%%%%%%%%%%%%%%%%%%%%%%%%%%%%%%

\section{Depth of Hopf subalgebras, modules and subgroup depth}  Let $ H$ be a finite-dimensional Hopf algebra over an arbitrary field $k$ with coproduct denoted by $\cop(h) = h\1 \otimes h\2$
(for each $h \in H$, an abbreviated version of Sweedler's $\sum_{(h)} h\1 \otimes h\2$). Let $R \subseteq H$ be a Hopf subalgebra, so $\cop(R) \subseteq R \otimes R$ and the antipode satisfies $S(R) = R$.   It was shown in \cite[Prop.\ 3.6]{K2013} that the tensor powers  of $H$ over $R$, denoted by $H^{\otimes_R n}$,  reduce to tensor powers of the generalized quotient $Q = H/ R^+H$ as follows: $H^{\otimes_R n} \stackrel{\cong}{\longrightarrow} H \otimes Q^{\otimes (n-1)}$ given by 
\begin{eqnarray}
\label{eq: arrow}
 x \otimes_R y \otimes_R \cdots \otimes_R z  & \mapsto &  xy\1\cdots z\1 \otimes \overline{y\2 \cdots z\2} \otimes \cdots \otimes \overline{z_{(n)}}.
\end{eqnarray}
This is an $H$-$H$-bimodule mapping where the right $H$-module structure on $H \otimes Q \otimes \cdots \otimes Q$ is given by the diagonal action of $H$: $(y \otimes v_1 \otimes \cdots \otimes v_{n-1}) \cdot h = yh\1 \otimes v_1 h\2 \otimes \cdots \otimes v_{n-1} h\n$.   This shows quite clearly
that the following definition will be of interest to computing $d(R,H)$.  Let $W$ be a right $H$-module and
$T_n(W) := W \oplus W^{\otimes 2} \oplus \cdots \oplus W^{\otimes n}$.  

\begin{definition}
A module $W$ over a Hopf algebra $H$ has \textbf{depth} $n$ if $T_{n+1}(W) \| q \cdot T_n(W) $
and depth $0$ if $W$ is isomorphic to a direct sum of copies of $k_{\eps}$, where $\eps$ is the counit.
Note that this entails that $W$ also has depth $n+1$, $n+2$, $\ldots$.  Let $d(W, \M_H)$ denote
its \textbf{minimum depth}.  If $W$ has a finite depth, it is said to be \textit{algebraic module}.  
  \end{definition}

The following lemma is worth noting here.  
\begin{lemma}
\label{lemma-modI}
Suppose a Hopf ideal $I$ is contained in the annihilator ideal of a module $W$ over a Hopf algebra $H$.
Then depth of $W$ is the same over $H$ or $H/I$.
\end{lemma}
\begin{proof}
The lemma is proven by noting that a Hopf ideal $I$ in $\Ann_H W$ is contained in the annihilator ideal of each tensor power of $W$,  since $I$ is a coideal (cf.\  Eq.~(\ref{eq: defi}) below).  Additionally, split epis as in $T_{n+1}(W) \| t \cdot T_n(W) $ descend and lift along $H \rightarrow H/I$.  
\end{proof}

Algebraic $H$-modules is a terminology consistent with algebraic module over group algebras
for the following reason.  Since $T_m(W) \| T_{m+1}(W)$, the indecomposable summands of $T_m(W)$
occur again (up to isomorphism) in the Krull-Schmidt decomposition of $T_{m+1}(W)$.   
If $W$ has depth $n$, all $T_m(W)$ and their summands $W^{\otimes m}$ for $m \geq n$
are expressible as sums of the indecomposable summands of $T_n(W)$.  This should be compared
to \cite[Chapter II.5.1]{Fe} to see that algebraic modules have finite depth and conversely; the proof
does not depend on the commutativity of the Green ring of a group algebra (see \cite[(1.3)]{CR}).  
Recall that the Green ring of a Hopf algebra $H$ over a field $k$, denoted by $A(H)$, is the free abelian group with basis consisting of  indecomposable $H$-module isoclasses, 
with addition given by direct sum, and the multiplication in its ring structure given by the tensor product.  For example,
$K_0(H)$ is a finite rank ideal in $A(H)$, since $P \otimes_k X$ is projective if  $P$ is projective and $X$ is a module (also a known fact for finite tensor categories \cite[Prop.\ 2.1]{EO}).  
As shown in \cite{Fe}, a finite depth $H$-module $W$ satisfies a polynomial with integer coefficients in $A(H)$, and conversely.  

\begin{example}
\begin{rm}
The paper \cite{C} mentions that the principal block of the simple group $M_{11}$ contains
$5$-dimensional simple modules that are not algebraic. 
\end{rm}
\end{example}

The main theorem in \cite[5.1]{K2013} proves from the basic Eq.~(\ref{eq: arrow}) that Hopf subalgebra (minimum) depth and depth of its generalized quotient $Q$ are closely related by
\begin{equation}
\label{eq: inequalityfordepth}
2d(Q,\M_R) + 1 \leq d(R,H) \leq 2d(Q, \M_R) + 2.
\end{equation}
Note that one restricts $Q$ to an $R$-module in order to obtain the better result on depth. In contrast minimum h-depth satisfies the equality, $  d_h(R,H) = 2d(Q, \M_H) + 1 $  \cite[5.1]{K2013}, but the interval in Eq.~(\ref{eq: in}) gives less precise information for ordinary depth.  

Next we combine the observations above with known results about algebraic modules of group algebras (see also \cite{Fe}).

\begin{prop}
\label{prop-berger}
Let $k$ be an algebraically closed field of characteristic $p > 0$.  Suppose $H$ is a finite-dimensional Hopf algebra with subgroup $G$ of grouplike elements that is solvable.  Let $Q = H/R^+H$ where $R = kG$.  If $Q$ is semisimple, then $R$ has finite depth in $H$.   
\end{prop}
\begin{proof}
The paper \cite{Be} shows that if $W$ is a simple $R$-module, it is an algebraic $R$-module.  But the isoclass of $Q$ is a sum of simples in the commutative $\Z$-algebra $A(R)$, where sums of algebraic elements are algebraic (see \cite[(1.5)]{CR}).  Then the statement follows from the inequality~(\ref{eq: inequalityfordepth}).  
\end{proof}
\subsection{Subgroup depth of a corefree subgroup pair}
\label{}
Recall that the core $\mbox{\rm Core}_G(H)$ of a subgroup pair of finite groups $H \leq G$ is the intersection of conjugate subgroups of $H$,; equivalently, the largest normal subgroup contained in $H$.
Note that if $N = \mbox{\rm Core}_G(H)$, then $\mbox{\rm Core}_{G/N}(H/N)$ is the one-element group, i.e., $H/N \leq G/N$ is a \textit{corefree} subgroup.  Theorem 6.9 in \cite{BKK} shows that if $N$ is the
intersection of $n$ conjugates of $H$, then $d(H,G) \leq 2n$; if moreover $N$ is contained in the center of $G$, $d(H,G) \leq 2n-1$.

In \cite[2.6, 6.8]{BKK} the following example of  minimum depth $4$ is observed. Let $G = S_4$, $H = D_8$, the dihedral group of
$8$ elements embedded in $G$, and $N = \{ (1), (12)(34), (13)(24), (14)(23) \}$.   Then $G/N \cong \Sigma_3$,
$H/N \cong S_2$, and minimum depth  satisfies  $d_0(G,H) = 4$ but  $d_0(H/N, G/N) = 3$.    Like combinatorial depth in 
\cite[Theorem 3.12(d)]{BDK}, subgroup depth satisfies the following inequality,  finite
groups $G \geq H$ with $N \leq \mbox{\rm Core}_G(H)$:
 \begin{prop}
\label{prop-core}
Suppose $N$ is a normal subgroup of a finite group $G$ contained in a subgroup $H \leq G$.  Then $$ d_0(H/N, G/N)\leq d_0(H,G) \leq d_0(H/N, G/N) +1.$$
\end{prop}
\begin{proof}
It is noted in \cite[6.8]{BKK} that for $k$ an algebraically closed field of characteristic zero, the inequality $d_k(H/N, G/N) \leq d_k(H,G)$ is satisfied.  But it is shown in \cite{BDK} that depth
of a finite group algebra extension is the same over any field of the same characteristic, in this case zero.  

The second inequality follows from the Inequality~(\ref{eq: inequalityfordepth}) (established in \cite[5.1]{K2013}), since the quotient module $Q \cong k[H \setminus G]$ by one of the Noether theorems, while $N$ acts trivially
on this $H$-module.  Note that $I = kHkN^+ = kN^+ kH$ is a Hopf ideal in $kH$ (generated by $\{ h - hn \| \, h \in H, n \in N \}$), which annihilates $Q$, and we apply Lemma~\ref{lemma-modI}. 
\end{proof}

Without the assumption of characteristic zero on a ground field $k$, we may use the second half of the proof to conclude that
with $d = d(Q, \M_{kG}) = d(Q, \M_{k[G/N]})$, the subgroup depths are in the following closed interval of length one:
\begin{equation}
\label{eq: inequality}
d_k(H,G), d_k(H/N, G/N) \in [2d+1, \,  2d+2].
\end{equation}

%%%%%%%%%%%%%%%%%%%%%%%%%%%%%%%%%%%%%%%%%%%%%%%%
\section{The descending chain of annihilators of the \\  tensor powers of $Q$}
\label{two}

In this section $H$ is a finite-dimensional Hopf algebra over a field $k$.  Let $R$ be a Hopf subalgebra
of $H$.  Let $H^+$ denote the kernel of the counit $\eps: H \rightarrow k$; then $R^+ = \ker \eps|_R$
is a coideal of $R$.  Recall that two right $H$-modules $U$ and $W$ have an $H$-module structure
on $U \otimes_k W$ from the diagonal action, $(u \otimes w) \cdot h = uh\1 \otimes wh\2$.  In this section we study the annihilator ideals of the tensor powers of the right $H$-module coalgebra $Q := H/ R^+H$
and its restriction to right $R$-module coalgebra.  The purpose for this is to obtain a necessary condition for finite depth of the subalgebra $R \subseteq H$.  For the convenience of the reader, we give several arguments that originated in  the pioneering \cite{R} and were illuminated by the related articles \cite{PQ, FK, CH}.
A useful fact for finite-dimensional Hopf algebras that we use below is that a bi-ideal $I$ of $H$
is automatically a Hopf ideal; i.e., if $I$ is an (two-sided) ideal and coideal of $H$, then it may be established that $S(I) = I$ for the antipode $S: H \rightarrow H$ (e.g., see \cite{PQ}).  

Given the right $R$-module $Q = H/R^+H$, its tensor powers $Q^{\otimes n} = Q \otimes \cdots \otimes Q$ ($n$ times $Q$) are also $R$-modules,  with annihilator ideals denoted by $I_n = \Ann_R Q^{\otimes n}$.  Thinking of the zeroeth power of $Q$ as the trivial $R$-module $k_{\eps}$,
denote $I_0 = R^+$.  Now if modules have a monic $U \into W$, one verifies that $\Ann \, W \subseteq \Ann \, U$.  Secondly,  the $R$-module coalgebra structure of $Q$ shows that
for each $n \geq 0$, $Q^{\otimes n} \| Q^{\otimes (n+1)}$ \cite[Prop.\ 3.8]{K2013}. It follows that
we have a descending chain of ideals,
\begin{equation}
\label{eq: dcc}
I_0 \supseteq I_1 \supseteq I_2 \supseteq \cdots \supseteq I_n := \Ann_R Q^{\otimes n} \supseteq \cdots
\end{equation}
In a moment we show in the proof of Lemma~\ref{lemma-Rief}  the (also known) fact that $I_n = I_{n+1}$  implies $I_n = I_{n+r}$ for
all positive integers $r$; in this case, if $\ell(R)$ denotes the length of $R$ as an $R^e$-module, the chain of ideals of $R$ in~(\ref{eq: dcc}) must
satisfy $I_n = I_{n+1}$ at some $n \leq \ell(R)$. Note that if $t$ is the number of nonisomorphic $R$-simples, then $\ell(R) \geq t$, with equality if and only if $R$ is semisimple \cite{FK}.  

\begin{example}
\begin{rm}
Suppose $I_0 = I_1$.  Then $R^+  \subseteq \Ann_R Q=  \{ r \in R^+ : Hr \subseteq R^+H \}$; i.e., $HR^+ \subseteq R^+H$, a condition that characterizes left ad-stable Hopf subalgebra as well as right depth two Hopf subalgebra \cite{BK}.  Thus,
$I_0 = I_1$ if and only if $R$ is a normal Hopf subalgebra in $H$ iff $d(R,H) \leq 2$.
\end{rm}
\end{example}

Let $I_Q  := \cap_{n=1}^{\infty} I_n$, an ideal in $R$; indeed $I_Q$ is the maximal Hopf ideal contained in $\Ann_R Q$, by the next lemma based on nice arguments given in \cite{R, PQ}, worth giving again in this context.
\begin{lemma}
\label{lemma-Rief}
Each Hopf ideal in $\Ann_R Q$ is contained in $I_Q$, which is itself a Hopf ideal.  Moreover,
$I_Q = I_n$ for some $n \leq \ell(R)$.  
\end{lemma}  
\begin{proof}
Suppose $I$ is a Hopf ideal in $\Ann_R Q$ and $x \in I$. Then $x$ annihilates $Q$, so that $(Q \otimes Q)\cdot x = (Q \otimes Q)\cop(x) = 0$
follows from the coideal property $\cop(I) \subseteq I \otimes R + R \otimes I$.  Similarly  $x \in I_n$ for all $n \geq 1$, since
the $n-1$'st power of (the coassociative) coproduct satisfies $\cop^{n-1}(x) \in I^{(n)}$, a subspace in $R^{\otimes n}$ defined generally by 
\begin{equation}
\label{eq: defi}
I^{(m+1)} := \sum_{i=0}^m R^{\otimes i} \otimes I \otimes R^{\otimes (m-i)}
\end{equation}
(which visibly annihilates $Q^{\otimes (m+1)}$).  

If $I_n = I_{n+1}$, we show $I_n = I_{n+2}$ and a similar induction argument shows that  $I_n = I_{n+r}$
for all $r \geq 0$.  If $x \in I_n = I_{n+1}$, then $\cop(x)$ annihilates $Q^{\otimes (n+1)} = Q^{\otimes n} \otimes Q$,
whence $\cop(x) \in I_n \otimes R + R \otimes I_1$.  Then $(\cop \otimes \id_R)\cop(x) \in
I_n \otimes R \otimes R + R \otimes I_1 \otimes R + R \otimes R \otimes I_1$, which
itself annihilates $Q^{\otimes n} \otimes Q \otimes Q = Q^{\otimes (n+2)}$.  Then $I_n = I_{n+2}$.  

From this it follows that $I_Q = \cap_{i=1}^n I_n = I_n$ and that $I_Q$ is a coideal.  For suppose
$x \in I_n = I_{2n}$.  Then $Q^{\otimes n} \cdot x = 0 = Q^{\otimes 2n}\cdot x$, so writing
$Q^{\otimes 2n} = Q^{\otimes n} \otimes Q^{\otimes n}$ shows that 
$x\1 \otimes x\2 \in I_n \otimes R + R \otimes I_n$, and thus $\cop(I_Q) \subseteq I_Q \otimes R + R \otimes I_Q$.  We conclude that $I_Q$ is a bi-ideal in $R$, whence a Hopf ideal, and the maximal Hopf ideal contained
in $I_1$.  Let $\ell_Q$ denote the least $n$ for which $I_Q = I_n$, so that $\ell_Q \leq \ell(R)$ follows
from the general remarks about composition series  following~(\ref{eq: dcc}).
\end{proof}

\begin{prop}
\label{prop-necessary}
If a Hopf subalgebra $R$ has depth $2n+2$ in a finite-dimensional Hopf algebra $H$, then $ \Ann_R Q^{\otimes n}\subseteq \Ann_R Q^{\otimes (n+r)} $
for all integers $r \geq 0$. 
\end{prop}
\begin{proof}
From the inequality~(\ref{eq: inequalityfordepth}), it follows that the depth
of $Q$ is $n$ or less (and $d(Q, \M_R) = n$ if $d(R,H) = 2n+2$ or $2n+1$).  Thus $Q^{\otimes (n+r)} \sim Q^{\otimes n}$ as $R$-modules, and these have equal annihilators.  That $I_{n+r} \subseteq I_n$ is always the case.
\end{proof}
Note that 
\begin{equation}
\Ann_R Q^{\otimes n} = \{ r \in R^+ \| \, H^{\otimes n}.r \in (R^+H)^{(n)} \} 
\end{equation}
from which it is possible to express the necessary condition for depth $2n+2$ in the proposition in continuation of the condition $HR^+ \subseteq R^+H$ for
depth $2$.  For example, denote $R^{++} := \{ r \in R^+ \| Hr \subseteq R^+H \}$; then a necessary condition that $R \subseteq H$ have depth $4$ is
\begin{equation}
(H \otimes H).R^{++} \subseteq (R^+H)^{(2)},
\end{equation}
which expresses that $\Ann_R Q \subseteq \Ann_R (Q \otimes Q)$.  
\begin{example}
\begin{rm}
Given a finite-dimensional Hopf algebra $H$ over an arbitrary field $k$ with radical ideal $J$,
the $H$-module $W = H/J$ may not be a coalgebra if $J$ fails to be a coideal. Of course $\Ann_H W = J$:  the annihilator ideals of $W^{\otimes n}$ are shown in \cite[Chen-Hiss]{CH} to satisfy $\Ann_H W^{\otimes n} = \bigwedge^n J$ (for the wedge product of subspaces of a coalgebra, see for example \cite[Chapter 5]{M}), which is also a descending series of ideals.  Therefore the lemma applies
to  $W = H/J$ as well, so the intersection $I_W$ of the annihilators of tensor powers of $W$ is the maximal nilpotent Hopf ideal $J_{\omega}$ in the radical of $H$, studied in \cite{CH}.  
For example, if $H$ has a projective simple, then $J_{\omega} = \{ 0 \}$ \cite[2.6(3)]{CH}
with a partial converse \cite[3.10]{CH} involving the condition $\ell_W \leq 2$. On the one hand, if the dual Hopf algebra $H^*$ is pointed, then $J_{\omega} = J$ \cite[Section 5.2]{M}; equivalently,  $H$
has the Chevalley property \cite{L} (i.e., tensor products of simple modules are semisimple). On the other hand, if    
$H = kG$ a group algebra over a field $k$ of characteristic $p$,  with normal Hopf subalgebra
$R = kO_p(G)$, the group algebra of the core $O_p(G)$ of a Sylow $p$-subgroup, then 
using \cite{PQ, CH} one notes that $J_{\omega}(H)$ is the Hopf ideal $R^+H = HR^+$.
It is verified in \cite[4.5]{CH} that for $k$ algebraically closed of characteristic $p \geq 5$, each of the nonabelian simple groups $G$ has a projective and simple $kG$-module (as suggested by the fact
that $O_p(G) = \{ 1 \}$).   
\end{rm}
\end{example}
Recall that an $R$-module $U$ is \textit{faithful} if $\Ann_R U = \{ 0 \}$.  
\begin{definition}
\label{def-cf}
Say that the quotient module $Q = H/ R^+H$ is conditionally faithful if  $I_Q = \{ 0 \}$, i.e., the annihilator ideal
$\Ann_R Q$ contains no nonzero Hopf ideal in $R$.  By Lemma~\ref{lemma-Rief} this implies
that $Q^{\otimes n}$ is faithful as an $R$-module for all $n \geq \ell_Q$. 
\end{definition}

It is well-known  that a finite-dimensional $R$-module $W$ is faithful if and only if $W$ is a generator.  For if $W$ is a generator, then for some $n \in \N$, there is $R_R \into n \cdot W$, whence $\Ann_R W \subseteq \Ann_R R = \{ 0 \}$.  Conversely, if $W$ is faithful,  define a monomorphism $R_R \into n \cdot W$ by $r \mapsto (w_1 r, \ldots,w_n r)$
where $w_1,\ldots,w_n$ is a $k$-basis of $W$.  Since $R$ is a (quasi-) Frobenius algebra, $R_R$ is an injective module, and the monomorphism just given is a
split monomorphism. Then $R_R \| n \cdot W_R$ and the  projective indecomposables (or principal modules) all divide
a faithful module $W$, as recorded below.  
\begin{lemma}
\label{lem-classic}
If $W_R$ is faithful, then each projective indecomposable $R$-module $P$ satisfies $P \| W$.  
\end{lemma}

\begin{example}
\label{ex-reg}
\begin{rm}
Let $R$ be a Hopf algebra where $\dim R \geq 2$.  Then the regular representation $R_R$
is faithful and projective, as are the tensor powers $R^{\otimes n}$ for integers $n \geq 1$. From the lemma it follows that
$R \sim R^{\otimes n}$ as $R$-modules, so that $\ell_R = 1$ and $d(R, \M_R) = 1$.  Similarly, a faithful projective $R$-module $W$  has depth $1$; a conditionally faithful projective $R$-module $Q$
has depth $\ell_Q$, as recorded next.  
\end{rm}
\end{example}
\begin{lemma}
\label{lem-referee}
Let $R$ be a Hopf algebra and $Q$ a conditionally faithful projective $R$-module.  Then $Q$ has
minimum depth $d(Q, \mathcal{M}_R) = \ell_Q$.  
\end{lemma}
\begin{proof}
Since the $Q^{\otimes (\ell_Q + r)}$ are faithful, projective $R$-modules for each integer $r \geq 0$, each contains as a summand every projective indecomposable by Lemma~\ref{lem-classic}.  Consequently, they are similar as $R$-modules:
$Q^{\otimes \ell_Q} \sim Q^{\otimes (\ell_Q + r)}$ for each $r \geq 0$.  It follows that
$d(Q,\M_R) \leq \ell_Q$. If $d = d(Q, \mathcal{M}_R)$, then
$Q^{\otimes d} \sim Q^{\otimes \ell_Q}$ is faithful, minimality of $\ell_Q$ forces $d = \ell_Q$.  
\end{proof}
\begin{theorem}
\label{theorem-Burn}
Suppose $R \subseteq H$ is a Hopf subalgebra with quotient module $Q$ a projective, conditionally faithful $R$-module.  
Then $R$ is semisimple, $\ell_Q \leq t$, where $t$ is the number of irreducible representations of $R$, and
each $R$-simple $S \| Q^{\otimes \ell_Q}$. Furthermore, the minimum depth satisfies $2\ell_Q + 1 \leq d(R,H) \leq 2\ell_Q + 2$.
\end{theorem}
\begin{proof}
If $Q = H / R^+H$ is a projective right $R$-module, then $R$ is semisimple \cite[3.5]{K2013}.
This may also be seen right away by noting that $k_R \| Q_R$, since the counit $\eps_Q: Q \rightarrow k$
is split by the mapping $\mu \mapsto 1\mu + R^+H$.  Then $k_R$ is projective, and $R$ is semisimple. 

Since $R$ is semisimple, the length $\ell(R)$ of $R_{R^e}$ satisfies $\ell(R) = t$; also, each projective indecomposable is a simple module and conversely.  Then $\ell_Q \leq t$ follows
from Lemma~\ref{lemma-Rief} , and each $S \|  Q^{\otimes \ell_Q}$ follows 
from Definition~\ref{def-cf} and Lemma~\ref{lem-classic}.

The last statement of the theorem follows from the inequality for depth Eq.~(\ref{eq: inequalityfordepth}) and Lemma~\ref{lem-referee}. 
\end{proof}
\begin{example}
\begin{rm}
Suppose $k = \C$ and the Hopf subalgebra $R$ is a group algebra $\C G$ where $G$ is a subgroup
of grouplike elements in a Hopf algebra $H$. Suppose that $Q = H/R^+H$ is conditionally faithful, 
then its character $\chi_Q$ is faithful, i.e., its kernel  $\ker \chi_Q = \{ g \in G | \chi_Q(g) = \chi_Q(1) \} = N$ is trivial, for if this normal subgroup were nontrivial, then $\Ann_R Q$ contains the nontrivial Hopf ideal $I
= R\C N^+ = \C N^+ R$.  Note that
if $\chi_Q(g) = \chi_Q(1)$, then $g$ acts like the identity on $Q$, whence $1-g \in \Ann_R Q$.  
Conversely, if the character $\chi_Q$ is faithful, the Burnside-Brauer Theorem \cite[p.\ 49]{I}
informs us that $Q$ is conditionally faithful, for $\chi_i \| \chi_Q^m$ for each irreducible character, $\chi_1, \ldots, \chi_t$ of $G$, and $m \leq |\chi_Q(G)|$, where $|X|$ denotes the cardinality of a finite set $X$. It follows that $\ell_Q \leq |\chi_Q(G) |$. (Alternatively for general $k$, if $Q_R$ is not conditionally faithful,
then $\Ann_R Q^{\otimes n}$ stabilizes as $n \rightarrow \infty$ on a nonzero Hopf ideal $I$ of the group algebra $R$ necessarily of the form $I = Rk N^+ = kN^+ R$ \cite{PQ, CH}, 
where $N$ is a normal subgroup of $G$ in $ \ker \chi_Q$.) 
\end{rm}
\end{example}

%%%%%%%%%%%%%%%%%%%%%%%%%%%%%%%%%%%%%%%%%%%%%%%
\section{Depth of a semisimple group algebra in its \\ Drinfeld double}
\label{three}
As an application of  Section~\ref{two} and the methods sketched in the last subsection of Section~\ref{one}, we compute the depth
of a semisimple group algebra in its Drinfeld double, a smash product of the group algebra and
its dual \cite{M}.   A certain portion of this section can be carried further to a general semisimple  or cocommutative Hopf algebra
in its Drinfeld double; the interested reader should first consult \cite{B} and \cite{Pa}. 

Suppose $G$ is a finite group, $k$  a field of characteristic not dividing the order of $G$, and consider the group algebra $R = k G$.  Denote its Drinfeld
double as $H = D(G) = D(R)$ \cite{M} with multiplication given by 
\begin{equation}
\label{eq: cocommutative}
( p_x \bowtie g)(p_y \bowtie h) = 
p_x p_{gyg^{-1}} \bowtie gh
\end{equation}
 for all $g,h,x,y \in G$ where $p_x$ denotes the one-point projection in $R^*$.  Note that
this is the semidirect product of the  $R$-module (adjoint representation) algebra $R^*$ with $kG$.  Recall that $1_H = \sum_{x \in G} p_x \bowtie 1_G$ and the counit $\eps(p_x \bowtie g)=
p_x(1_G) = \delta_{x,1}$.  Of course $R$ is identifiable with the subalgebra $1_{R^*} \otimes R$.
A short computation with Eq.~(\ref{eq: cocommutative}) shows that the centers of $D(G)$ and $G$ satisfy
\begin{equation}
\label{eq: center}
kZ(G) = Z(D(G)) \cap kG.
\end{equation}

We compute the generalized quotient $Q = H/R^+H$ as a right $R$-module.  Note that $\dim Q = |G|$.  

\begin{lemma}
The right $G$-module $Q$ is isomorphic to $k G_{\rm ad}$.  
\end{lemma}
\begin{proof}
First compute $R^+H$ from
$$( 1_H \bowtie (1-g))( p_y \bowtie h) = p_y \bowtie h - p_{gyg^{-1}} \bowtie gh, $$
for each $1 \neq g, y, h \in G$.  Thus in $H/R^+H$ the cosets have a unique representative as follows:
$$\overline{p_y \bowtie h} = \overline{p_{gyg^{-1}} \bowtie gh} = \overline{p_{h^{-1}yh} \bowtie 1_G} $$ 

Define a $G$-module isomorphism $Q \stackrel{\cong}{\longrightarrow} R^*$ by $\overline{p_y \bowtie h} \mapsto
p_{h^{-1} y h}$.  But ${k G^*}_{\rm ad} \cong k G_{\rm ad}$ via $p_g \mapsto g$,
where the right adjoint is given by $g \cdot x = x^{-1} g x$.  
\end{proof}
It is well-known that in characteristic zero, $D(R)$ is a semisimple algebra, if $R$ is semisimple
\cite[2.5.2, 10.3.13]{M}.  
\begin{prop}[Burciu \cite{BP}]
\label{theorem-china}
The module $Q = kG_{\rm ad}$ has depth $n$ if the $kG$-module $Q^{\otimes n}$ is
faithful for some $n \in \N$.  A converse to this requires $k$ to be an algebraically closed field of $\mbox{\rm char}\, k  = 0$ and that $G$ has trivial center:  if $kG \subseteq D(G)$ has depth $2n+1$, then $Q^{\otimes n}$ is faithful.  
\end{prop}
\begin{proof}
($\Leftarrow$)
Since $kG$ is a semisimple algebra, the $kG$-modules $Q$ and its tensor powers are projective (and semisimple) modules. 
Thus if $Q^{\otimes n}$ is faithful, then $Q$ is conditionally faithful with $\ell_Q \leq n$.
But minimum depth $d(Q, \mathcal{M}_G) = \ell_Q$ by Lemma~\ref{lem-referee}. 

($\Rightarrow$)  Use the relation $\sim$ between simple $kG$-modules $W,U$ 
defined by $W \sim U$ if $W \otimes _R H$ and $U \otimes_R H$ have an isomorphic nonzero summand in common \cite[p. 139]{BKK}.  (In terms of the bipartite graph of the semisimple subalgebra pair
$R \subseteq H$, the points representing $W$ and $U$ are connected by one irreducible representation of $H$.)  Extend $\sim$ by transitive closure to an equivalence relation.  Note that $\sim$ is already
a transitive relation iff $R \subseteq H$ has depth $3$ \cite[Corollary 3.7]{BKK}.  Also,
the number of equivalence classes is equal to $\dim Z(H) \cap R$ \cite[Corollary 3.3]{BKK}, so by
the hypothesis and Eq.~(\ref{eq: center}) there is one equivalence class.  

Let $W$ be a left $R$-module (and note that the ${}_R\M$ is isomorphic as tensor categories to $\M_R$ via
 the inverse).  We compute $W\uparrow^{D(R)} \downarrow_R$ from
$$ R^* \otimes_k R \otimes_R W \cong R^* \otimes_k W$$ with $G$-action given by
$g \cdot p_x \otimes w = p_{gxg^{-1}} \otimes gw$.  This implies that the image of $W$ under induction and restriction satisfies
\begin{equation}
\label{eq: adjoin}
W\uparrow^{D(R)} \downarrow_R \cong {}_{\rm ad}R \otimes W,
\end{equation}
 the right-hand side having the diagonal action by $R$.

Let $\chi_U$ denote the character of a $G$-module $U$,  $\chi_{\rm ad}$
be the character of module ${}_{\rm ad}R$, and $\chi_1,\ldots,\chi_t \in \mbox{\rm Irr}(G)$.  If  $R \subseteq H$
has depth $3$, then $\sim$ has one equivalence class, so that the inner product of any irreducible characters, $\chi_U, \chi_W$ of $G$, satisfies
$\bra \chi_U \uparrow^{D(G)}, \chi_W \uparrow^{D(G)} \ket > 0 $.  By Frobenius reciprocity and
Eq.~(\ref{eq: adjoin}) this gives $\bra \chi_U, \chi_{\rm ad} \chi_W \ket > 0$, so letting 
$\chi_W = \chi_k$, this shows that  ${}_{\rm ad}R$ and $ R_{\rm ad}$ are generators, therefore  faithful modules.

If $R$ in $H$ has depth $5$, then by \cite[Proposition 5.4]{BKK},  any two $R$-simples $U,W$  may be connected by a shortest path of length  at most two,  $
U \sim X \sim W$ for some $R$-simple $X$, and that the entry $\bra \chi_U, \chi_{\rm ad}^2 \chi_W \ket  > 0$ in $\mathcal{S}^2$ (where $\mathcal{S}$ is the symmetric order $t$ matrix defined in Section~\ref{one} by $\mathcal{S}_{ij} = \bra \chi_i\uparrow^{D(G)},\chi_j \uparrow^{D(G)} \ket$). It follows that $Q^{\otimes 2}$ is faithful.  The rest of the proof is a similar induction argument using \cite[Prop.\ 5.4]{BKK}.  
\end{proof}
Recall from Section~\ref{two} that $Q$ is conditionally faithful if $\Ann_R Q^{\otimes \ell_Q} = \{ 0 \}$
for some $\ell_Q \geq 1$, while $\Ann_R Q^{\otimes m} \neq \{ 0 \}$ for $0 \leq m < \ell_Q $. 
\begin{cor}
\label{cor-Pass}
Suppose $k$ is an algebraically closed field of characteristic zero and $G$ is a finite, centerless group.  Then adjoint module $Q$ is conditionally faithful and its minimum depth as an $kG$-module is $\ell_Q$
\end{cor}
\begin{proof}
From the hypotheses on $k$, it follows from \cite{BKK} that $kG \subseteq D(G)$ has a finite depth.
Suppose it has depth $2n+1$; then by the proposition, $Q^{\otimes n}$ is a faithful $kG$-module.  
It follows that $n \geq \ell_Q = d(Q, \mathcal{M}_R)$ by Lemma~\ref{lem-referee}.  
 \end{proof}
As we will see in Corollary~\ref{cor-precisedetermination} the minimum depth is in fact satisfying $$d(\C G, D(G)) = 2\ell_Q + 1.$$    
\begin{example}
\begin{rm}
Let $k$ be a field of characteristic zero. 
The paper \cite[Theorem 1.10]{Pa} shows that for each $n \geq 3$ the symmetric group $S_n$  has a faithful adjoint action on $kS_n$.  It follows from Corollary~\ref{cor-Pass} that $3 \leq d(kS_n, D(S_n)) \leq 4$ (in fact $d(kS_n,D(S_n)) = 3$ follows from Theorem~\ref{th-young} below).  

Note that $d(kS_n, D(S_n)) = 3$ for specific $n = 3,4, \ldots$ also follows from a computation that the symmetric matrix $\mathcal{S} > 0$, i.e., has all positive entries.  In general the methods above are realized from the $r \times r$ character table $(\chi_i(g_j))$ of a group $G$ with values in $\C$ as follows.  The character
$\chi_{\rm ad}$ is given by row vector $(| C_G(g_j) | )_{j=1,\ldots,r}$, where an entry is the number of elements of the centralizer subgroup of $g_j$. The inner product $\bra \chi_{\rm ad}, \chi_j \ket$ is the sum $\sum_{i=1}^r \chi_j(g_i)$;
e.g. $\bra \chi_{\rm ad}, \chi_1 \ket = r$, the number of orbits of the permutation module by Burnside's Lemma \cite{I}.  That no row of the character table sums to zero is then equivalent to the module $\C G_{\rm ad}$ being  faithful.  Also  the center of $G$ equals the kernel of $\chi_{\rm ad}$, and is trivial
if no $g \neq 1$ satisfies $\chi_{\rm ad}(g) = \chi_{\rm ad}(1) = |G|$. 
\end{rm}
\end{example}

%%%%%%%%%%%%%%%%%%%%%%%%%%%%%%%%%%%%%%%%%%%%%%%
\section{On depth of a Hopf algebra in a smash product}
\label{five}

In this section we show that a Hopf algebra $H$ has finite depth in its smash product algebra $A \# H$
if the  left $H$-module algebra  $A$ is an algebraic $H$-module.  

Suppose $H$ is a Hopf algebra and $A$ is a left $H$-module algebra.  Recall that equations such as
$h . 1_A = \eps(h)1_A$ and $h. (ab) = (h\1 . a)(h\2 . b)$ are satisfied ($a,b \in A$, $h \in H$):  briefly, $A$ is an algebra in the tensor category of left $H$-modules.  
 Define the smash product (or cross product \cite{Ma})
by $A \# H = A \otimes H$ as a linear space with multiplication given by 
\begin{equation}
\label{eq: smashmultiplication}
(a \# h)(b \# k) = a (h\1 . b) \# h\2 k
\end{equation}
Notice how $H$ identifies with the subalgebra $1_A \# H$ in $A \# H$ and if $a = 1_A$, the
action of $h$ is the diagonal action. 

\begin{prop}
\label{prop-young}
The $n$-fold tensor powers of $A \# H$ over $H$ are isomorphic as $H$-$H$-bimodules
to the following tensor products in the tensor category ${}_H\M$:  
\begin{equation}
\label{eq: smashoverHopf}
 (A \# H)^{\otimes_H n} \cong A^{\otimes n} \otimes H
\end{equation}
\end{prop}  
\begin{proof}
The case $n = 1$ follows from the mapping $a \# h \mapsto a \otimes h$, which is clearly right $H$-linear and also left $H$-linear by  
an application of Eq.~(\ref{eq: smashmultiplication}).

Suppose Eq.~(\ref{eq: smashoverHopf}) holds for an $H$-$H$-bimodule isomorphism for $1 \leq n < m$.  Since $H \otimes_H A \cong A$, it follows from induction that
$$ (A \# H)^{\otimes_H m} \cong (A \# H)^{\otimes_H (m-1)} \otimes_H A \# H \cong $$
$$ A^{\otimes (m-1)} \otimes H \otimes_H A \otimes H \cong A^{\otimes m} \otimes H.$$

Note that the isomorphism becomes
$
a \# u \otimes_H b \# v \otimes_H \cdots \otimes_H c \# w$
\begin{equation}
\label{eq: explicit}
  \longmapsto a \otimes u\1. b \otimes  \cdots \otimes u_{(n-1)}v_{(n-2)} \cdots . c \otimes u\n v_{(n-1)}\cdots w
\end{equation}
 for $u,v,w \in H$ and $a,b,c \in A$.  
\end{proof}
Define the  minimum odd depth of a subalgebra $B \subseteq A$ as $d_{\mathrm{odd}}(B,A) = 2\lceil \frac{d(B,A)-1}{2} \rceil +1$, which  is the least odd integer greater than or equal to the minimum depth $d(B,A)$.

\begin{theorem}
\label{th-young}
The minimum odd depth of a finite-dimensional Hopf algebra in its smash product satisfies
\begin{equation}
\label{eq: young}
d_{\rm odd}(H, A \# H) = 2 d(A, {}_H\M) + 1
\end{equation} 
\end{theorem}
\begin{proof}
Since $A$ is a left $H$-module algebra, it follows from applying any of the standard face and degeneracy
mappings, which are $H$-module maps, that $A^{\otimes m} \| A^{\otimes (m+1)}$
for each integer $m \geq 0$.  Then the depth $n$ condition for the left $H$-module $A$
given by $T_{n+1}(A) \| q \cdot T_n(A)$ for some $q \in \N$ is equivalent to
$A^{\otimes (n+1)} \| q \cdot A^{\otimes n}$ for some $q \in \N$.  Tensoring this by $- \otimes H$
yields $A^{\otimes (n+1)} \otimes H \| q \cdot A^{\otimes n} \otimes H$ and thus by Proposition~\ref{prop-young}
$(A \# H)^{\otimes_H (n+1)} \| q \cdot (A \# H)^{\otimes_H n}$ as $H$-$H$-bimodules.
Thus the minimum odd depth  $d_{\rm odd}(H, A \# H) \leq 2d(A, {}_H\M)+1$ by Definition~\ref{def-depth}.  

Conversely, if $(A \# H)^{\otimes_H (n+1)} \| q \cdot (A \# H)^{\otimes_H n}$ as $H$-$H$-bimodules,
we apply Proposition~\ref{prop-young} and write equivalently $A^{\otimes (n+1)} \otimes H \| q \cdot A^{\otimes n} \otimes H$.  Next apply 
$ - \otimes {}_Hk$ to this, and through the cancellation ${}_HH \otimes {}_Hk \cong {}_Hk$ with the  unit module in ${}_H\M$, we obtain $A^{\otimes (n+1)} \| q \cdot A^{\otimes n}$, which is the depth $n$ condition for an $H$-module algebra.  Therefore $2d(A,{}_H\M) + 1 \leq d_{\rm odd}(H, A \# H)$. The conclusion of the theorem follows from the two inequalities established.  
\end{proof}
\begin{cor}
\label{cor-equivalence}
Let $R$ be a Hopf subalgebra of a finite-dimensional Hopf algebra and $Q$ its right quotient $H$module
coalgebra.  Then we have 
\begin{equation}
d(R,H) - d(R, Q^* \# R) \leq 2
\end{equation}
and 
\begin{equation}
d_h(R,H) = d_{\mathrm{odd}}(H, Q^* \# H).
\end{equation}
\end{cor}
\begin{proof}
The proof follows from Theorem~\ref{th-young} by an argument
given in the \cite[Cor.\ 5.5]{K2013}.
\end{proof}
\begin{cor}
\label{cor-precisedetermination}
The subalgebra depth and the depth
of $Q = kG_{\rm ad}$ are related by $d_{\rm odd}(kG, D(G)) = 2 d(Q, \M_{kG}) + 1$.  If $k$ is algebraically closed and has characteristic $0$ and the center of $G$ is trivial, then 
 $Q$ is conditionally faithful and  the depth satisfies $d(kG,D(G)) = 2\ell_Q +1$.
\end{cor}
\begin{proof}
First note from Eq.~(\ref{eq: cocommutative}) that $D(G) \cong (kG)^* \# kG$ where the action is the
adjoint action, ${}_{\rm ad}kG^*$, which is isomorphic to $Q$.  Then Eq.~(\ref{eq: young}) implies
that $d_{\rm odd}(kG,D(G)) = 2d(Q, \M_{kG}) + 1$. 

For the second statement,  note that Corollary~\ref{cor-Pass} shows that $d(Q, \M_{kG}) = \ell_Q$.     From the inequality~(\ref{eq: inequalityfordepth}) depth of the centerless group algebra in its Drinfeld double satisfies 
$d(kG,D(G)) = 2\ell_Q+1$ or $2\ell_Q +2$;  if $d(kG,D(G)) = 2\ell_Q +2$, then $d_{\rm odd}(kG,D(G)) = 2\ell_Q +3$.  But Theorem~\ref{th-young} then implies that $d(Q, \M_{kG}) = \ell_Q + 1$,
a contradiction. 
\end{proof}

\begin{example}
\label{ex-Pass}
\begin{rm}

The minimal example suggested in \cite[Lemma 1.3]{Pa} for a centerless group $G$ with adjoint action on $\C G$
that is not faithful, is a semidirect product $G$ of a rank $3$ elementary $3$-group with the Klein $4$-group, so that $|G| = 108$ \cite{BP}.  A long computation by hand (albeit brief with a computer program) of its order $15$ character table and $\mathcal{S}$-matrix (where $\mathcal{S}_{ij} = \bra \chi_i, \chi_{\rm ad} \chi_j \ket$) shows that $\mathcal{S}$ has zero entries, but $\mathcal{S}^2 > 0$, whence there is $q \in \N$ such that $\mathcal{S}^3 \leq q \mathcal{S}^2$.  It follows from the sketch in Section 1 of results in \cite{BKK} and  Corollary~\ref{cor-precisedetermination} that 
the minimum depth satisfies $d(\C G, D(G)) = 5$.  Hence $\ell_Q = 2$ for $Q = \C G_{\rm ad}$.  
\end{rm}
\end{example}

\begin{example}
\label{ex-heisenbergdouble}
\begin{rm}
Let $H$ be a Hopf algebra of dimension $n \geq 2$.  Let $H^*$ act on $H$ by $f \rightharpoonup h =
h\1 f(h\2)$.  It is a standard check that $H$ is a left $H^*$-module algebra.  Their smash product
$H \# H^*$ is the \textit{Heisenberg double} of $H$ \cite[Ch.\ 9]{M}. 
We compute the
depth $d_{\rm odd}(H^*, H \# H^*)$ next from $d(H, {}_{H^*}\M)$ and Theorem~\ref{th-young}. 
Since $H^*$ is a Frobenius algebra, ${}_{H^*}H \cong {}_{H^*}H^*$ is isomorphic to the regular representation of $H^*$.  It was noted in Example~\ref{ex-reg} that $d(H, {}_{H^*}\M) = 1$.
It follows that \begin{equation}
\label{eq: heisenberg}
d_{\rm odd}(H^*, H \# H^*) = 3.
\end{equation}  

$$ \begin{array}{rlcrl}
& &\stackrel{n}{\bullet} &  &  \\
&&&& \\
    & /   &     | & \setminus &  \\
&&&& \\
\stackrel{n_1}{\circ} & & \cdots  & & \stackrel{n_t }{\circ} 
\end{array} $$

This result on depth makes good sense, since $H \# H^* \cong M_n(k)$ via the (Galois) algebra isomorphism
$\lambda: H \# H^* \stackrel{\cong}{\longrightarrow} \End_k H$ given by $\lambda(h \# f)(x) = h(f \rightharpoonup x)$. Thus $H \# H^*$ is an Azumaya $k$-algebra; then 
$H^* \into H \# H^*$ is an H-separable extension if the extension is split and projective (cf.\ \cite{LK2011}). 
In this case $d_h(H^*, H \# H^*) = 1$ and $d(H^*, H \# H^*) = 2$.  If $H^*$ is a semisimple complex algebra, that $2 = $ $d(H^*, H \# H^*) $ may also be seen from the bipartite graph of the inclusion \cite{BKK} pictured above (where $n_1,\ldots, n_t$ denote the dimensions of the simples of
$H^*$).  
\end{rm}
\end{example}
\subsection{Concluding remarks}
A similar result may be obtained for the depth of a cocommutative finite-dimensional Hopf algebra $H$ embedded
as a Hopf subalgebra in the tensor Hopf algebra $H \otimes H$ via the coproduct $\cop: H \rightarrow H \otimes H$.  The generalized quotient module $Q = H \otimes H/ \cop(H^+)H \otimes H
\stackrel{\cong}{\longrightarrow} H_{\rm ad}$ via $\overline{x \otimes y} \mapsto S(x)y$.
Let $H_{\rm ad}$ denote the right (adjoint) $H$-module algebra, with smash product $H \# H_{\rm ad}$ given by $(x \# y)(z \# w) = xz\1 \# (y\cdot z\2) w$, where $y \cdot z = S(z\1)yz\2$. 
One checks that $\cop: H \rightarrow H \otimes H \cong H \# H_{\rm ad}$ forms a commutative
triangle with respect to the mappings $x \# y \mapsto x\1 \otimes x\2 y$ and the identification 
$h \mapsto h \# 1_H$ (cf. \cite[7.3.3]{M}).  If $W_H$ is a module, its induced-restricted module
$$W \otimes_{\cop(H)} H \otimes H \cong W \otimes_H H \otimes H_{\rm ad} \cong W \otimes H_{\rm ad}$$ shows that the corresponding characters satisfy $\chi_W\uparrow^{H\otimes H} \downarrow_{\cop(H)} = \chi_W \chi_{\rm ad}$.  These steps may be followed in an alternative fashion 
(using \cite[4.21 and 5.1]{I}) for the characters of a finite group $G$
and its embedding $g \mapsto (g,g)$ in $G \times G$. Also one shows by a short argument
that the center of $kG \otimes kG$ (for $k$ any ground field) and the center $Z(G)$ satisfy $Z(kG \otimes kG) \cap \cop(kG) = \cop(kZ(G))$. 

  If the ground field
$k$ is algebraically closed,  $H$ is pointed, therefore isomorphic as algebras to the smash product
$H_1 \# kG$ where $G$ is the group of grouplike elements of $H$ and $H_1$ is a connected Hopf subalgebra of $H$ \cite[5.6.4]{M}.   If $k$ is moreover a field of characteristic zero, then $H_1$
is an enveloping algebra for the primitive elements of $H$ \cite[5.6.5]{M}; therefore $H_1 = k$ by the
finite-dimensional hypothesis on $H$, and $H = kG$, a group algebra. 
 The steps
followed in Sections~4, 5, and~6  show that if $G$ has trivial center, then 
$d(H, H\otimes H) = 2 \ell_{H_{\rm ad}} + 1$; e.g., $d_{\C}(G, G \times G) = 2 \ell_{\C G_{\rm ad}} +1$.

%%%%%%%%%%%%%%%%%%%%%%%%%%%%%%%%%%%%%%%%%%%%%%%%5
\subsection{Acknowledgements}   The authors thank Sebastian Burciu for scientific exchanges related to Sections~\ref{three} by email in 2011 and in Porto May 2012.
The second author thanks Martin Lorenz for discussions in Philadelphia in 2012.  Research for this paper was funded by the European Regional Development Fund through the programme {\tiny COMPETE} 
and by the Portuguese Government through the FCT  under the project 
\tiny{ PE-C/MAT/UI0144/2011.nts}.

\end{document}